\documentclass[arxiv=true]{agn_article}

\usepackage[markexamples=true,markremarks=true]{agn_all}
\usepackage{caption}
\usepackage{multicol}

\newcommand{\imin}[1][\lambda]{m{\ifx&#1&\else(#1)\fi}}
\newcommand{\imax}[1][\lambda]{M{\ifx&#1&\else(#1)\fi}}
\newcommand{\iset}[1][\lambda]{J{\ifx&#1&\else(#1)\fi}}
\renewcommand{\lambdahat}{\widehat{\smash{\lambda}\vphantom{2}}}
\renewcommand{\lambdatilde}{\widetilde{\lambda}}
\DeclareMathOperator{\CSO}{CSO}
\newcommand{\genCSO}{Z}
\newcommand{\samples}{343}

\providecommand{\citet}[2][]{\citeauthor{#2} \cite[#1]{#2}}

\begin{document}
\title{Rank-one convexity vs.\ ellipticity for isotropic functions}
\date{\today}
\knownauthors[martin]{martin,voss,ghiba,neff}
\maketitle
\begin{abstract}
\noindent
	It is well known that a twice-differentiable real-valued function $W\col\GLpn\to\R$ on the group $\GLpn$ of invertible \nnmatrices\ with positive determinant is \emph{rank-one convex} if and only if it is \emph{Legendre-Hadamard elliptic}. Many energy functions arising from interesting applications in isotropic nonlinear elasticity, however, are not necessarily twice differentiable everywhere on $\GLpn$, especially at points with non-simple singular values.
	
	Here, we show that if an isotropic function $W$ on $\GLpn$ is twice differentiable at each $F\in\GLpn$ with simple singular values and Legendre-Hadamard elliptic at each such $F$, then $W$ is already rank-one convex under strongly reduced regularity assumptions. In particular, this generalization makes (local) ellipticity criteria accessible as criteria for (global) rank-one convexity to a wider class of elastic energy potentials expressed in terms of \emph{ordered} singular values. Our results are also directly applicable to so-called \emph{conformally invariant} energy functions. We also discuss a classical ellipticity criterion for the planar case by Knowles and Sternberg which has often been used in the literature as a criterion for global rank-one convexity and show that for this purpose, it is still applicable under weakened regularity assumptions.
\end{abstract}
\\[.65em]
{%
	\tableofcontents%
}

\section{Introduction}
\label{section:introduction}
In the context of nonlinear elasticity theory, we consider the deformation $\varphi\col\Omega\subset\R^n\to\R^n$ of an elastic body $\Omega\subset\R^n$. For so-called hyperelastic material models, the elastic behaviour of the body is determined by an energy potential function $W\col\GLpn\to\R,\, F\mapsto W(F)$ depending on the deformation gradient $F=\grad\varphi$. Since elastic deformations are assumed not to be self-intersecting, the natural domain of such an energy function is given by the group $\GLpn$ of invertible \nnmatrices\ with positive determinant.

In this paper, we will consider the following two specific properties of energy functions. %

\begin{definition}
\label{definition:rankOneConvexity}
	A function $W\col\GLpn\to\R$ is called \emph{rank-one convex} if for all $F\in\GLpn$ and all $H\in\Rnn$ with $\rank(H)=1$ and $F+H\in\GLpn$, the mapping $t\mapsto W(F+tH)$ is convex on the interval $[0,1]$ or, equivalently,
	\begin{equation}\label{eq:rankOneConvexityDefinition}
		\begin{aligned}
			W((1-t)F_1+tF_2) \leq (1-t)\.W(F_1) + t\.W(F_2)
			\quad&\text{ for all }\; t\in[0,1]\,,\; F_1,F_2\in\GLpn
			\\&\text{ with } \rank(F_2-F_1)=1\,.
		\end{aligned}
	\end{equation}
\end{definition}
\begin{definition}
\label{definition:ellipticity}
	A function $W\col\GLpn\to\R$ is called \emph{Legendre-Hadamard elliptic}, \emph{LH-elliptic} or simply \emph{elliptic} at $F\in\GLpn$ if $W$ is two-times differentiable and
	\[
		D^2W[F].(\xi\otimes\eta,\xi\otimes\eta)\geq 0\qquad\text{for all }\;\xi,\eta\in\R^n\,.
	\]
\end{definition}

It is well known and easy to show (cf.\ Proposition \ref{proposition:ellipticityRankOneEquivalence}) that rank-one convexity and ellipticity of a sufficiently regular energy function are, in fact, equivalent \cite{Dacorogna08}. It is, however, interesting to note that within the context of nonlinear elasticity, the two properties mainly arise in rather different fields: On the one hand, rank-one convexity is deeply connected to the (purely \emph{mathematical}) investigation of the \emph{existence of minimizers} for energy functionals of the form
\[
	I\col\mathcal{A}\to\R\,,\qquad I(\varphi) = \int_\Omega W(\grad\varphi(x))\,\dx
\]
within an appropriate set $\mathcal{A}$ of admissible functions. This connection is mostly due to the relation between rank-one convexity and other generalized convexity properties; most importantly, rank-one convexity of $W\col\GLpn\to\R$ is a necessary condition for $W$ to be quasiconvex \cite{morrey1952quasi} or polyconvex \cite{ball1976convexity,ball1977constitutive,agn_schroder2010poly}.%

Legendre-Hadamard ellipticity, on the other hand, is usually considered a \emph{constitutive requirement} for an elastic energy potential, i.e.\ a property of a material model motivated by \emph{mechanical} considerations. In particular, ellipticity plays an important role for material stability \cite{knowles1976failure,knowles1978failure,silhavy1997mechanics,grabovsky2018rank} and ensures finite wave propagation speed in elastic bodies \cite{eremeyev2007constitutive,zubov2011,sawyersRivlin78}. LH-ellipticity is also directly connected to other constitutive properties of materials; for example (cf.\ Proposition \ref{proposition:knowlesSternbergOriginal}), every elliptic isotropic energy function $W$ also satisfies the \emph{tension-extension inequalities} \cite{truesdell65}
as well as the \emph{Baker-Ericksen inequalities} \cite{bakerEri54}, cf.\ Section \ref{section:BEinequalities}. The notion of LH-ellipticity is applicable to more general models of solid deformations as well \cite{agn_neff2014loss,agn_shirani2020legendre}.

\section{Rank-one convexity, ellipticity and regularity of isotropic functions}

The well-known equivalence of ellipticity and rank-one convexity for two-times differentiable functions is easy to establish.
\begin{proposition}
\label{proposition:ellipticityRankOneEquivalence}
	Let $W\col\GLpn\to\R$ be two-times differentiable. Then $W$ is rank-one convex if and only if $W$ is elliptic at every $F\in\GLpn$.
\end{proposition}
\begin{proof}
	Let $W$ be elliptic, and let $F\in\GLpn$ and $H\in\Rnn$ with $\rank(H)=1$. Then there exist $\xi,\eta\in\R^n$ such that $H=\xi\otimes\eta$, thus
	\begin{equation}\label{eq:positiveSecondDerivativeRankOneDirection}
		\ddtsq\; W(F+tH) = D^2W[F+tH].(H,H) = D^2W[F+tH].(\xi\otimes\eta,\xi\otimes\eta) \geq 0\,.
	\end{equation}
	Thus the mapping $t\mapsto W(F+tH)$ is two-times differentiable on $[0,1]$ with nonnegative second derivative, which implies that the mapping is convex.
	
	Similarly, if $W$ is rank-one convex, then for given $\xi,\eta\in\R^n\setminus\{0\}$, we find $\rank(H)=1$ for $H\colonequals\xi\otimes\eta$. Since $\GLpn\subset\Rnn$ is open, we can choose $\eps>0$ sufficiently small such that $F-\frac\eps2H\in\GLpn$ and $F+\frac\eps2H\in\GLpn$. Then the rank-one convexity of $W$ implies that the mapping $t\mapsto W(F-\frac\eps2H+t\.\eps H)$ is convex on $[0,1]$ and thus
	\[
		0 \leq \ddtsq\; W(F-\tfrac\eps2 H+t\eps H)\big|_{t=\frac12} = D^2W[F].(\eps\.H,\eps\.H) = \eps^2\,D^2W[F].(\xi\otimes\eta,\xi\otimes\eta)\,. \qedhere
	\]
\end{proof}

Since ellipticity is a purely local property \cite{agn_ghiba2015ellipticity} which can in many cases be checked by direct computational means, the equivalence stated in Proposition \ref{proposition:ellipticityRankOneEquivalence} is often used to decide whether a given energy function is rank-one convex. For the planar isotropic case, in particular, a number of criteria for ellipticity in terms of the singular values of the deformation gradient $F$ are available, cf.\ Section \ref{section:knowlesSternberg}.

However, in many interesting applications in isotropic nonlinear elasticity, the requirement of $C^2$-regularity is far too strict. In particular, isotropic energy functions are often most naturally expressed in terms of the \emph{ordered singular values} of the deformation gradient, i.e.\ in the form
\begin{equation}\label{eq:energyInTermsOfOrderedSingularValues}
	W\col\GLpn\to\R\,,\quad W(F) = \ghat(\lambdahat_1,\dotsc,\lambdahat_n)
\end{equation}
for all $F\in\GLpn$ with singular values $\lambdahat_1\geq\ldots\geq\lambdahat_n$, where $\ghat\col\Oset\to\R$ is a real-valued function defined on the set
\[
	\Oset \colonequals \{(x_1,\dotsc,x_n)\in\Rp^n \setvert x_1\geq\ldots\geq x_n\}\,.
\]
Although, in applications, such a function is generally sufficiently regular on the interior (cf.\ Fig.~\ref{figure:orderedVectorSet})
\[
	\intOset = \{(x_1,\dotsc,x_n)\in\Rp^n \setvert x_1>\ldots> x_n\}
\]
and usually at least differentiable up to the boundary of $\Oset$ relative to $\Rp^n$, the function $W$ defined by \eqref{eq:energyInTermsOfOrderedSingularValues} is not necessarily differentiable (even once) in that case. This loss of regularity occurs at the points $F\in\GLpn$ where the singular values of $F$ are non-simple, corresponding to the boundary points of $\Oset$.

On the other hand, for given $\ghat\col\Oset\to\R$, let $g\col\Rp^n\to\R$ denote the uniquely defined symmetric (that is invariant under permutations of the arguments) function such that $g(x_1,\dotsc,x_n)=\ghat(x_1,\dotsc,x_n)$ for all ordered vectors $(x_1,\dotsc,x_n)\in M$. Then
\begin{equation}\label{eq:energyInTermsOfUnorderedSingularValues}
	W(F) = g(\lambda_1,\dotsc,\lambda_n)
\end{equation}
for all $F\in\GLpn$ with (not necessarily ordered) singular values $\lambda_1,\dotsc,\lambda_n$, and the regularity of $g$ directly corresponds to that of the energy $W$; more specifically, if $g$ is a function of class $C^k$, then so is $W$ and vice versa (cf.\ Corollary \ref{corollary:singularValuesAnalyticFunctions} as well as \cite[Theorem 6.4]{ball1984differentiability} and \cite{vsilhavy2000differentiability}).
It is easy to see that if $\ghat$ is $k$--times (continuously) differentiable at $(\lambda_1,\dotsc,\lambda_n)\in\intOset$, then so is $g$ at any permutation of $(\lambda_1,\dotsc,\lambda_n)$ and thus $W$ at any $F\in\GLpn$ with singular values $\lambda_1,\dotsc,\lambda_n$. The regularity of $g$ at any ordered $(\lambda_1,\dotsc,\lambda_n)$ with $\lambda_i=\lambda_{i+1}$ for some $i\in\{1,\dotsc,n\}$, however, requires additional conditions on the function $\ghat$ which are generally not satisfied.\footnote{%
	In the context of criteria for classical convexity (cf.\ \cite{hill1970}) of a real-valued function $W$ on $\GLpn$, the global differentiability of $W$ was called \enquote{[\ldots] surprisingly tedious to verify} by Ball \cite[p.~363]{ball1976convexity}.%
}

This loss of regularity, of course, means that LH-ellipticity is generally not well defined at any $F\in\GLpn$ with non-simple singular values. Therefore, it is difficult to establish global rank-one convexity of such an energy function by considering its pointwise ellipticity.

\begin{example}
\label{example:operatorNormEnergy}
	As a simple example, consider the function $W\col\GLp(n)\to\R$ with
	\begin{equation}\label{eq:exampleOperatorNormEnergy}
		W(F) \colonequals \opnorm{F} = \max\{\lambda_1,\dotsc,\lambda_n\} \equalscolon g(\lambda_1,\dotsc,\lambda_n)
	\end{equation}
	for all $F\in\GLpn$ with singular values $\lambda_1,\dotsc,\lambda_n$, where $\opnorm{\,.\,}$ denotes the operator norm. Then $W$ is a convex function and therefore rank-one convex. However, $W$ is not twice differentiable (not even in rank-one directions) on $\GLpn$ since, for example,
	\[
		W(c\.\id + t\.\diag(1,0,\dotsc,0)) = \max\{c+t,c,\dotsc,c\} =
		\begin{cases}
			c &: t\leq0\\
			c+t &: t>0
		\end{cases}
	\]
	for any $c>0$. Therefore, $W$ is not elliptic everywhere in the sense of Definition \ref{definition:ellipticity}, hence it would be impossible to establish the rank-one convexity of $W$ by any criteria which rely on pointwise ellipticity alone. Note that $W$ is indeed twice differentiable and LH-elliptic at every $F\in\GLpn$ with only simple singular values and that the representation
	\[
		\ghat\col\Oset\to\R\,,\quad \ghat(\lambdahat_1,\dotsc,\lambdahat_n) = \lambdahat_1
	\]
	in terms of the ordered singular values $\lambdahat_1\geq\dotsc\geq\lambdahat_n$ of $F$ can be extended to an analytic function $\gtilde\col\Rp^n\to\R$ by letting $\gtilde(\lambda_1,\dotsc,\lambda_n)=\lambda_1$, but that this extension is not equal to the representation $g$ of $W$ given in \eqref{eq:exampleOperatorNormEnergy}.
\end{example}

Note carefully that although the set of matrices with non-simple singular values is a nowhere dense set of measure zero, it is not at all obvious whether ellipticity on its complement is sufficient for rank-one convexity without any additional assumptions; even in the one-dimensional case, where rank-one convexity is equivalent to classical convexity and ellipticity corresponds to a non-negative second derivative, a function $W\col\GLp(1)\cong\Rp\to\R$ might be non-convex even if its restriction to both $(0,x_0]$ and $[x_0,\infty)$ is smooth for some $x_0\in\Rp$ with $W''(x)>0$ for all $x\neq x_0$. A simple, highly symmetric  example is shown in Fig.\ref{figure:nonRegularExampleGraph}.

\begin{figure}[h]
	\begin{center}
		\begin{minipage}[t]{.609\textwidth}
			\begin{center}
				\begin{tikzpicture}
					\begin{axis}[
						axis x line=bottom,axis y line=left,
						x label style={%
							at={(ticklabel* cs:.98)},
							above right
						},
						xlabel={$x$},
						xmin=-.833, xmax=.833,
						ymin=-0.0343, ymax=.343,
						xtick={0},
						ytick=\empty,
						xticklabels={$x_0$},
						yticklabels=\empty,
						hide obscured x ticks=false,
				        width=.9898\linewidth,
				        height=.49\linewidth
					]
						\addplot[domain=-1.015:0,samples=\samples,thick,black] {0.49*(x+.4998)^2}
							node[pos=.833, below right] {$W$};
						\addplot[domain=0:1.015,samples=\samples,thick,black] {0.49*(x-.4998)^2};
						\addplot[domain=0:.49,samples=\samples,thick,blue,dashed] {0.49*(x+.4998)^2}
							node[pos=.343, below right] {$W_-$};
						\addplot[domain=0:-.49,samples=\samples,thick,red,dashed] {0.49*(x-.4998)^2}
							node[pos=.343, below left] {$W_+$};
					\end{axis}
				\end{tikzpicture}
			\par\end{center}
			\vspace*{-1.47em}
			\caption{\label{figure:nonRegularExampleGraph}%
				Example of a non-convex, continuous function $W$ with $W''(x)>0$ for all $x\neq x_0$ such that $W\big|_{(0,x_0]}$ and $W\big|_{[x_0,\infty})$ are smooth, i.e.\ can be extended to smooth functions $W_-$ and $W_+$ on $\Rp$, respectively.%
			}
		\end{minipage}
		\hfill
		\begin{minipage}[t]{.343\textwidth}
			\begin{center}
				\begin{tikzpicture}
					\begin{axis}[
						axis x line=middle,axis y line=middle,
        				x label style={at={(current axis.right of origin)},anchor=north, below},
						xlabel={\hspace*{-.637em}$\lambda_1$},
						ylabel={$\lambda_2$},
						xmin=-.049, xmax=1.05,
						ymin=-.049, ymax=1.05,
						xtick=\empty,
						ytick=\empty,
				        width=.9898\textwidth,
				        height=.9898\textwidth
					]
						\addplot[thick,opacity=0,fill=blue, fill opacity=0.098]coordinates {(1.001, 1.001) (1.001, 0) (0, 0)}
							node[pos=0.5,color=black,opacity=1,above left,xshift=-0.8cm,yshift=0.8cm] {$\intOset[2]$};
						\addplot[domain=0:1.001,samples=2,thick,red] {x}
							node[pos=.637, above left] {$\partial\Oset[2]$};
					\end{axis}
				\end{tikzpicture}	
			\par\end{center}
			\vspace*{-1.47em}
			\caption{\label{figure:orderedVectorSet}%
				Visualization of the set $\Oset[2]$; in the two-dimensional case, the boundary relative to $\Rp^2$ is given by $\partial\Oset[2]=\{(\lambda,\lambda)\setvert\lambda>0\}$.%
			}
		\end{minipage}
		\hfill
	\end{center}
\end{figure}

However, in the scalar case, the following well-known condition based on the one-sided derivatives of a function is sufficient to ensure that convexity still holds if regularity is lost only at discrete points.

\begin{lemma}[{\cite[p.~35]{hiriart2013convex}}]
\label{lemma:convexityOneSidedDerivatives}
	Let $W\col I\to\R$ be a continuous real-valued function on an interval $I\subset\R$ such that for finitely many $t_1,\dotsc,t_m\in I$,
	\begin{itemize}
		\item[i)]
			$W$ is two-times differentiable on $I\setminus\{t_1,\dotsc,t_m\}$,
		\item[ii)]
			$W''(t)\geq0$ for all $t\in I\setminus\{t_1,\dotsc,t_m\}$,
		\item[iii)]
			for each $i\in\{1,\dotsc,m\}$, the \emph{left and right one-sided derivatives}
			\begin{equation}\label{eq:oneSidedDerivativeDefinition}
				\partial^- W(t_i) \colonequals \lim_{h\upto0} \frac{W(t_i+h)-W(t_i)}{h}
				\qquad\text{and}\qquad
				\partial^+ W(t_i) \colonequals \lim_{h\downto0} \frac{W(t_i+h)-W(t_i)}{h}
			\end{equation}
			exist and satisfy $\partial^- W(t_i)\leq \partial^+ W(t_i)$.
	\end{itemize}
	Then $W$ is convex on $I$. Furthermore, if $W$ is a convex function on $I$, then $\partial^- W(t)$ and $\partial^+ W(t)$ are well-defined and $\partial^- W(t_i)\leq \partial^+ W(t_i)$ for every $t\in I$.
	\directqed
\end{lemma}

In Sections \ref{section:planarCase} and \ref{section:generalCase}, we
apply Lemma \ref{lemma:convexityOneSidedDerivatives} to deduce the rank-one convexity of suitable functions on $\GLpn$ from their ellipticity at each $F\in\GLpn$ with only simple singular values. Under appropriate assumptions, it turns out that the additional requirement posed on the one-sided derivatives is already implied by the so-called Baker-Ericksen inequalities.

\subsection{The Baker-Ericksen inequalities}
\label{section:BEinequalities}

It is well known \cite{marsden1994foundations} that if an objective-isotropic energy $W\col\GLpn\to\R$ is elliptic at $F\in\GLpn$ with singular values $\lambda_1,\dotsc,\lambda_n$, then the \emph{Baker-Ericksen inequalities}
\begin{alignat}{3}
\label{eq:BEinequalities}
	\frac{\lambda_i\.\pdd{g}{\lambda_i}-\lambda_j\.\pdd{g}{\lambda_j}}{\lambda_i-\lambda_j}
	\;&\geq0\;
	&\qquad\tforall i,j\in\{1,\dotsc,n\}
	&\quad\twith \lambda_i\neq\lambda_j
\intertext{%
	hold at $(\lambda_1,\dotsc,\lambda_n)$, where $g\col\Rp^n\to\R$ denotes the unordered singular value representation of $W$, cf.\ \eqref{eq:energyInTermsOfUnorderedSingularValues}. In terms of the ordered singular value representation $\ghat\col\Oset\to\R$ of $W$ given in \eqref{eq:energyInTermsOfOrderedSingularValues}, the Baker-Ericksen inequalities \eqref{eq:BEinequalities} can equivalently be expressed as
}%
\label{eq:BEinequalitiesOrdered}
	\lambdahat_i\.\pdd{\,\ghat}{\lambdahat_i}
	\;&\geq\; \lambdahat_j\.\pdd{\,\ghat}{\lambdahat_j}
	&\qquad\tforall i,j\in\{1,\dotsc,n\}
	&\quad\twith i< j
	\,.
\intertext{%
	Assume that $\ghat$ is continuously differentiable on $\Oset$ up to the boundary%
	\footnotemark
	$\partial\Oset$ relative to $\Rp^n$, i.e.\ that $\ghat$ can be extended to (or, equivalently, is a restriction of) a continuously differentiable function on $\Rp^n$, and that $\ghat$ satisfies \eqref{eq:BEinequalitiesOrdered} on $\intOset$. Then due to the assumed regularity of $\ghat$, the inequalities \eqref{eq:BEinequalitiesOrdered} hold on $\partial\Oset$ as well. In particular, for $(\lambdahat_1,\dotsc,\lambdahat_n)\in\partial\Oset$ with $\lambdahat_i=\lambdahat_j=\lambda$ for all $i,j\in\{\imin[],\dotsc,\imax[]\}\subset\{1,\dotsc,n\}$, we find
}%
\label{eq:BEinequalitiesOrderedNonSimpleSingularValues}
	\lambda\.\pdd{\,\ghat}{\lambdahat_i}
	\;&\geq\; \lambda\.\pdd{\,\ghat}{\lambdahat_j}
	&\qquad\tforall i,j\in\{\imin[],\dotsc,\imax[]\}
	&\quad\twith i< j
	\,.
\end{alignat}%
\footnotetext{%
	Recall that $\partial\Oset$ consists of the ordered vectors in $\Rp^n$ for which not all of the entries are distinct.%
}%
In this case,
\begin{equation}\label{eq:partialDerivativeOrderingImpliedByBEinequalities}
	\pdd{\,\ghat}{\lambdahat_{\imin[]}}
	\;\geq\; \ldots
	\;\geq\; \pdd{\,\ghat}{\lambdahat_{\imax[]}}
	\,;
\end{equation}
in other words, the partial derivatives %
of $\ghat$ (or, more precisely, of the extension of $\ghat$ to $\Rp^n$) with respect to components which are equal at $(\lambdahat_1,\dotsc,\lambdahat_n)\in\partial\Oset$ are in descending order. This property holds, in particular, for any $\ghat$ with $\ghat(\lambdahat(F))=W(F)$ if $W$ is elliptic at each $F\in\GLpn$ with simple singular values.
In the following, this observation will allow us to omit any additional conditions on $\ghat$ at the boundary of $\Oset$.

\section{Rank-one convexity and ellipticity in the planar isotropic case}
\label{section:planarCase}
For now, we will focus on the planar case $n=2$. Note that in this case, $F\in\GLp(2)$ has non-simple singular values $\lambda_1=\lambda_2=a$ if and only if the singular value decomposition of $F$ is of the form
\[
	F = Q_1\.\matr{\lambda&0\\0&\lambda}\.Q_2 = \lambda\.Q_1\.Q_2
\]
with $Q_1,Q_2\in\SO(2)$. In particular, $F$ has non-simple singular values if and only if $F$ is \emph{conformal}, i.e.\ if and only if there exist $\lambda\in\Rp$ and $Q\in\SO(2)$ such that $F=\lambda\cdot Q\in\CSO(2)$, where $\CSO(2)=\Rp\cdot\SO(2)$ is the conformal special orthogonal group. We will also denote the ordered singular values $\lambdahat_1(F)\geq\lambdahat_2(F)$, respectively, by $\lambdamax(F),\lambdamin(F)$ or simply by $\lambdamax,\lambdamin$.

As an important example, we first consider the class of so-called \emph{conformally invariant} energy functions, i.e.~ any $W\col\GLp(2)\to\R$ with $W(\genCSO F)=W(F\genCSO)=W(F)$ for all $\genCSO\in\CSO(2)$. These energy functions, which play an important role in the theory of conformal and quasiconformal mappings \cite{grotzsch1928einige,teichmuller1944verschiebungssatz,astala2008elliptic} as well as nonlinear elasticity \cite{iwaniec2009,agn_voss2019volIsLog,agn_hartmann2003polyconvexity}, can be expressed in terms of the so-called \emph{linear distortion function}, given by
\[
	K\col\GLp(2)\to\R\,,\qquad K(F) = \frac{\opnorm{F}^2}{\det F} = \frac{\lambdamax}{\lambdamin}\,,
\]
where $\opnorm{F}=\lambdamax$ denotes the operator norm of $F$. Note that the mapping $F\mapsto K(F)$ is itself rank-one convex (and even polyconvex) on $\GLp(2)$ due to the convexity of the mapping $(X,\delta)\mapsto\frac{\opnorm{X}^2}{\delta}$ on $\R^{2\times2}\times\Rp$.

\begin{example}
\label{example:conformallyInvariantEnergies}
	Let $W\col\GLp(2)\to\R$ be conformally invariant. Then there exists a function $\hhat\col[1,\infty)\to\R$ such that \cite{agn_martin2015rank}
	\begin{equation}\label{eq:conformallyInvariantEnergies}
		W\col\GLp(2)\to\R\,,\qquad W(F) = \hhat(K(F)) = \hhat \left( \frac{\lambdamax}{\lambdamin} \right) \equalscolon \ghat(\lambdamax,\lambdamin)\,.
	\end{equation}
	For the representation $\ghat\col\Oset[2]=\{(x_1,x_2)\in\Rp^2 \setvert x_1\geq x_2\}\to\R$ of $W$ in terms of ordered singular values, we find $\ghat\in C^k(\Oset)$ if and only if $\hhat\in C^k([1,\infty))$. Now, let
	\begin{align}
		h\col(0,\infty)\to\R\,,\qquad
		h(t) =
		\begin{cases}
			\hhat(t) &: t\geq1\\
			\hhat\left(\frac1t\right) &: t<1
		\end{cases}\,.
	\end{align}
	Then
	\[
		W(F) = \hhat \left( \frac{\max\{\lambda_1,\lambda_2\}}{\min\{\lambda_1,\lambda_2\}} \right) = h \left( \frac{\lambda_1}{\lambda_2} \right) \equalscolon g(\lambda_1,\lambda_2)%
	\]
	for all $F\in\GLp(2)$ with singular values $\lambda_1,\lambda_2$ in arbitrary order. In particular, the common regularity condition $g\in C^k(\Rp^2)$ for the representation $g$ of $W$ in terms of unordered singular values corresponds to the requirement $h\in C^k((0,\infty))$ which, in addition to $\hhat\in C^k([1,\infty))$, poses additional conditions on the derivatives of $\hhat$ at $1$, most notably $\hhat'(1)=0$ for $k\geq1$. Even in the simple case $\hhat(t)=t$ corresponding to the energy expression $W(F)=K(F)$, this condition is obviously not satisfied (cf.\ Fig.~\ref{figure:conformalNonRegularityGraph}). Therefore, criteria for rank-one convexity which are based purely on classical LH-ellipticity are not applicable to many practically relevant cases of conformally invariant energy functions.
\end{example}

It has recently been shown \cite{agn_martin2015rank,agn_martin2019envelope} (cf.\ \cite{agn_ghiba2015exponentiated}) that an energy of the form \eqref{eq:conformallyInvariantEnergies} is rank-one convex on $\GLp(2)$ if and only if it is polyconvex, which is the case if and only if $\hhat$ is monotone increasing and convex, regardless of any regularity. However, applications in nonlinear elasticity generally require these conformally invariant (or \emph{isochoric}) energies to be coupled with a \emph{volumetric} term in order to accurately model the behaviour of an elastic material. If, for example, an energy with an additive volumetric-isochoric split of the form
\[
	W\col\GLp(2)\to\R\,,\qquad W(F) = \hhat(K(F)) + f(\det F)
\]
with a function $f\col(0,\infty)\to\R$ is considered, then the equivalence of rank-one convexity and polyconvexity no longer holds \cite{agn_voss2019volIsLog}, and neither convexity condition implies or is implied by the simultaneous convexity and monotonicity of $\hhat$.

Again, in such cases, criteria for (pointwise) ellipticity might not be applicable in order to establish rank-one convexity due to a lack of regularity. However, the following result reduces the required regularity assumptions for such applications and shows that ellipticity of $W$ at each $F\in\GLp(2)$ with simple singular values is already sufficient for $W$ to be (globally) rank-one convex on $\GLp(2)$.

\begin{theorem}
\label{theorem:mainResultPlanar}
	Let $W\col\GLp(2)\to\R$ be an objective and isotropic function with
	\[
		W(F) = \ghat(\lambdamax(F),\lambdamin(F))
	\]
	for a real-valued function $\ghat\col\Oset[2]\to\R$ on the set $\Oset[2]=\{(x_1,x_2)\in\Rp^2\setvert x_1\geq x_2\}$, where $\lambdamax(F)\geq\lambdamin(F)$ are the singular values of $F$. If $\ghat\in C^2(\intOset[2])\cap C^1(\Oset[2])$ and $W$ is Legendre-Hadamard elliptic at each $F\in\GLp(2)$ with simple singular values $\lambdamax(F)\neq\lambdamin(F)$, then $W$ is rank-one convex on $\GLp(2)$.
\end{theorem}
\begin{remark}
	The requirement that $W$ is elliptic at each $F$ with simple singular values can equivalently be expressed as the ellipticity of $W$ on the set $\GLp(2)\setminus\CSO(2)$, where $\CSO(2)=\Rp\cdot\SO(2)$ denotes the special conformal orthogonal group.
\end{remark}
\begin{remark}
	The regularity assumption $\ghat\in C^1(\Oset[2])$ requires $\ghat$ to have any differentiable extension to $\Rp^2$, which is a significantly weaker requirement than $g\in C^1(\Rp^2)$ for the symmetrical extension $g\col\Rp^2\to\R$ with $g(y,x)=g(x,y)=\ghat(x,y)$ for all $x\geq y$, cf.\ Examples \ref{example:operatorNormEnergy} and \ref{example:conformallyInvariantEnergies}.
\end{remark}
\begin{remark}
	A criterion for rank-one convexity equivalent to Theorem \ref{theorem:mainResultPlanar} expressed purely in terms of the singular value representation $\ghat$ will be given in Proposition \ref{proposition:globalknowlesSternbergImprovement}.
\end{remark}
In order to prove Theorem \ref{theorem:mainResultPlanar}, we will require some basic properties of the singular value mapping $F\mapsto(\lambdamax(F),\lambdamin(F))$, starting with the following well-known regularity result.

\begin{lemma}[{\cite[Theorem 1]{de1989analytic} (cf.\ \cite[Theorem 5]{de2000multilinear} and \cite{vallee2006convex}}]
\label{lemma:singularValuesAbsoluteValueAnalyticFunctions}
	Let $X\col I\to\Rnn$ be an analytic, matrix-valued function on an interval $I\subset\R$. Then there exist analytic functions $\mu_1,\dotsc,\mu_n\col I\to\R$ such that for each $t\in I$, the singular values of $X(t)$ are given by $\abs{\mu_1(t)},\dotsc,\abs{\mu_n(t)}$.
\end{lemma}
If $\det(X(t))\neq0$ for all $t\in I$, then we can assume without loss of generality that the (nonzero) singular values of $X(t)$ are given directly by the analytic functions $\mu_i$ in Lemma \ref{lemma:singularValuesAbsoluteValueAnalyticFunctions}.
\begin{corollary}
\label{corollary:singularValuesAnalyticFunctions}
	Let $F,H\in\Rnn$, and let $I\subset\R$ be an interval such that $\det(F+tH)\neq0$ for all $t\in I$. Then there exist analytic functions $\mu_1,\dotsc,\mu_n$ such that for each $t\in I$, the singular values of $F+tH$ are given by $\mu_1(t),\dotsc,\mu_n(t)$.
\end{corollary}
The condition $\det(F+tH)\neq0$ is, in particular, satisfied for $t\in[0,1]$ if $F\in\GLpn$ and $\rank(H)=1$ with $F+H\in\GLpn$ due to the rank-one convexity of the set $\GLpn$. Note carefully that the mappings $\mu_i$ are not necessarily ordered on the interval $I$ (see Figure \ref{figure:singularValuesGraph}) and that the mapping $t\mapsto \lambdahat(F+tH)$ of $t$ to the ordered vector $\lambdahat(F+tH)$ of singular values of $F+tH$, although continuous, may indeed be non-differentiable on $I$ if non-simple singular values occur.

In the planar case, this specific situation of \enquote{intersecting} singular values, which is visualized in Fig.\ \ref{figure:singularValuesGraph}, allows for a direct comparison of the left and right derivatives of the ordered singular values $\lambdamax\geq\lambdamin$.

\begin{figure}[h]
	\begin{center}
		\begin{minipage}[t]{.4655\textwidth}
			\begin{center}
				\begin{tikzpicture}
					\begin{axis}[
						axis x line=bottom,axis y line=left,
						x label style={%
							at={(ticklabel* cs:.98)},
							above right
						},
						xlabel={$t$},
						xmin=-0.07, xmax=3.43,
						ymin=-0.07, ymax=3.43,
						xtick={1.001},
						ytick=\empty,
						xticklabels={$1$},
						yticklabels=\empty,
						hide obscured x ticks=false,
				        width=.9898\textwidth,
				        height=.7\textwidth
					]
						\addplot[domain=.98:-0.07,samples=\samples,thick,blue,dashed] {x}
							node[pos=.343, below right] {$\htilde$};
						\addplot[domain=0.245:1.001,samples=\samples,thick,red] {1/x}
							node[pos=.735, above right] {$h$};
						\addplot[domain=1.001:3.43,samples=\samples,thick,black] {x}
							node[pos=.49, below right] {$\hhat=h$};
					\end{axis}
				\end{tikzpicture}
			\par\end{center}
			\vspace*{-1.47em}
			\caption{\label{figure:conformalNonRegularityGraph}%
				The mapping $t\mapsto h(t)=\max\{t,\frac1t\}$ is not differentiable on $(0,\infty)$, whereas $\hhat\in C^\infty([1,\infty))$ with $\hhat(t)=t$ can be extended to a smooth function $\htilde$ on $(0,\infty)$.%
			}
		\end{minipage}
		\hfill
		\begin{minipage}[t]{.4655\textwidth}
			\begin{center}
				\hspace*{-1.47em}
				\begin{tikzpicture}
					\begin{axis}[
						axis x line=bottom,axis y line=left,
						x label style={%
							at={(ticklabel* cs:.98)},
							above right
						},
						xlabel={$t$},
						xmin=-.833, xmax=.833,
						ymin=-0.49, ymax=.49,
						xtick={0},
						ytick={0},
						xticklabels={$0$},
						yticklabels={$c$},
						hide obscured x ticks=false,
				        width=.9898\textwidth,
				        height=.7\textwidth
					]
						\addplot[domain=-1.015:0,samples=\samples,thick,blue] {sqrt(x+1.47)-sqrt(1.47)}
							node[pos=.343, below right] {$\lambdamin=\mu_2$};
						\addplot[domain=0:1.015,samples=\samples,thick,black] {sqrt(x+1.47)-sqrt(1.47)}
							node[pos=.637, above left] {$\lambdamax=\mu_2$};
						\addplot[domain=-1.015:0,samples=\samples,thick,black] {0.49*((x-.49)^2-.49^2)}
							node[pos=.735, below left = -0.347em] {$\lambdamax=\mu_1$};
							
						\addplot[domain=0:1.015,samples=\samples,thick,blue] {.49*((x-.49)^2-.49^2)}
							node[pos=.539, below] {$\lambdamin=\mu_1$};
					\end{axis}
				\end{tikzpicture}
			\par\end{center}
			\vspace*{-1.47em}
			\caption{\label{figure:singularValuesGraph}%
				Visualization of the singular values of $F+tH$ for $F=c\cdot Q$ with $c>0$ and $Q\in\SO(2)$, showing two analytic curves $\mu_1$ and $\mu_2$ intersecting at $t=0$; note that $\partial^-\lambdamax=\partial^+\lambdamin$ and $\partial^+\lambdamax=\partial^-\lambdamin$ at $t=0$.%
			}
		\end{minipage}
	\end{center}
\end{figure}
\begin{lemma}
\label{lemma:singularValueDerivativePropertiesPlanar}
	Let $F\in\GLp(2)$ with singular values $\lambda_1=\lambda_2=c$ and $H\in\R^{2\times2}$. %
	Then the mappings
	\[
		t\mapsto\lambdamax(t)\colonequals\lambdamax(F+tH)
		\qquad\text{and}\qquad
		t\mapsto\lambdamin(t)\colonequals\lambdamin(F+tH)
	\]
	are both left-differentiable and right-differentiable in a neighbourhood of $t=0$ with
	\begin{equation}\label{eq:planarOneSidedDerivativesEqualityInequality}
		\partial^-\lambdamax(0)=\partial^+\lambdamin(0)
		\leq
		\partial^-\lambdamin(0)=\partial^+\lambdamax(0)\,.
	\end{equation}
\end{lemma}
\begin{proof}

	Let $\mu_1,\mu_2$ be the singular value functions given by Corollary \ref{corollary:singularValuesAnalyticFunctions}.
	Since $\mu_1(0)=\mu_2(0)=c$, and since the roots of the analytic function $\mu_1-\mu_2$ are isolated unless $\mu_1\equiv\mu_2$, we can assume that there exists $\eps>0$ such that $\mu_1(t)\geq\mu_2(t)$ for all $t\in(-\eps,0]$. Then either $\mu_1(t)\geq\mu_2(t)$ for all sufficiently small $t>0$ as well, in which case the mappings $t\mapsto\lambdamax(t)=\mu_1(t)$ and $t\mapsto\lambdamin(t)=\mu_2(t)$ are both differentiable at $t=0$ and thus \eqref{eq:planarOneSidedDerivativesEqualityInequality} is trivially satisfied; or $\mu_1(t)<\mu_2(t)$ for all sufficiently small $t>0$, in which case
	\[
		\lambdamax(t) =
		\begin{cases}
			\mu_1(t) &: t\leq0\\
			\mu_2(t) &: t\geq0
		\end{cases}
		\qquad\text{and}\qquad
		\lambdamin(t) =
		\begin{cases}
			\mu_2(t) &: t\leq0\\
			\mu_1(t) &: t\geq0
		\end{cases}
		\;.
	\]
	Then the one-sided derivatives of $\lambdamax$ and $\lambdamin$ are well defined with
	\[
		\partial^-\lambdamax(0) = \mu_1'(0) = \partial^+\lambdamin(0)
		\qquad\text{and}\qquad
		\partial^+\lambdamax(0) = \mu_2'(0) = \partial^-\lambdamin(0)
		\,.
	\]
	Finally, $\lambdamax(0)=c=\lambdamin(0)$ and $\lambdamin(t)\leq\lambdamax(t)$ for all $t>0$, which directly yields the remaining inequality $\partial^+\lambdamin(0)\leq\partial^+\lambdamax(0)$.
\end{proof}

The next lemma, which follows directly from a more general characterization of rank-one connectedness with applications to multi-well problems in the calculus of variations \cite[Lemma 8.25]{rindler2018calculus}, states that there are no \emph{rank-one connected} matrices in the conformal special orthogonal group $\CSO(2)=\Rp\cdot\SO(2)$. The subsequent corollary ensures that under the regularity assumptions of Theorem \ref{theorem:mainResultPlanar}, the mapping $t\mapsto W(F+tH)$ is, for any $F\in\GLp(2)$ and any admissible rank-one direction $H\in\R^{2\times2}$, indeed twice differentiable on $(0,1)\setminus\{t_0\}$ for some $t_0\in(0,1)$.
\begin{lemma}[{\cite[Lemma 8.25]{rindler2018calculus}}]
\label{lemma:CSOnoRankOneConnections}
	Let $\genCSO_1,\genCSO_2\in\CSO(2)$. Then $\rank(\genCSO_2-\genCSO_1)\neq1$.
\end{lemma}
\begin{corollary}
\label{corollary:singleNonSimplePoint}
	Let $F\in\GLp(2)$ and $F\in\R^{2\times2}$ such that $\rank(H)=1$ and $F+H\in\GLp(2)$. Then there exists no more than one $t\in[0,1]$ such that $F+tH$ has non-simple singular values.
\end{corollary}
\begin{proof}
	Recall that $\genCSO\in\GLp(2)$ has non-simple singular values if and only if $\genCSO\in\CSO(2)$. Now, if both $\genCSO_1\colonequals F+t_1H$ and $\genCSO_2\colonequals F+t_2H$ for $t_1\neq t_2$, then $\genCSO_1,\genCSO_2\in\CSO(2)$ with
	\[
		\rank(\genCSO_2 - \genCSO_1) = \rank((t_2-t_1)\.H) = \rank(H) = 1\,,
	\]
	in contradiction to Lemma \ref{lemma:CSOnoRankOneConnections}.
\end{proof}
\medskip\noindent
We can now proceed to the proof of our main result for the planar case.
\begin{proof}[Proof of Theorem \ref{theorem:mainResultPlanar}]
	We need to show that under the stated assumptions, the mapping $t\mapsto W(F+tH)$ is convex on $[0,1]$ for any $F\in\GLp(2)$ and any $H\in\R^{2\times2}$ with $\rank(H)=1$ and $F+H\in\GLp(2)$. If the singular values of $F+tH$ are simple for all $t\in(0,1)$, then the regularity and ellipticity of $W$ directly yield (cf.\ \eqref{eq:positiveSecondDerivativeRankOneDirection})
	\[
		\ddtsq\; W(F+tH) = D^2W[F+tH].(H,H) \geq 0
	\]
	and thus the convexity of the mapping. Otherwise, due to Corollary \ref{corollary:singleNonSimplePoint}, there exists a unique $t_0\in(0,1)$ such that $F+t_0H$ has non-simple singular values or, equivalently, $F+t_0H\in\CSO(2)$. After re\-parameterization and rescaling of $H$, it therefore remains to show that for every $c>0$ and $Q\in\SO(2)$, the mapping
	\[
		p\col[-1,1]\to\R\,,\qquad p(t) = W(c\.Q+tH) = \ghat(\lambdamax(t),\lambdamin(t))
	\]
	is convex, where we use the notation $\lambdamax(t)=\lambdamax(c\.Q+tH)$ and $\lambdamin(t)=\lambdamin(c\.Q+tH)$.
	
	Again, due to the assumption of ellipticity on $\GLp(2)\setminus\CSO(2)$, we find $p''(t)\geq0$ for all $t\neq0$. According to Lemma \ref{lemma:convexityOneSidedDerivatives}, we thus only need to show that $\partial^+ p(0)\geq\partial^- p(0)$. Using Lemma \ref{lemma:singularValueDerivativePropertiesPlanar}, we compute
	\begin{align}
		\partial^+ p(0) &= \partial^+ \ghat(\lambdamax(t),\lambdamin(t)) \big|_{t=0}\notag\\
		&= \pdd{\,\ghat}{\lambdamax}(\lambdamax(0),\lambdamin(0)) \cdot \partial^+ \lambdamax(0) + \pdd{\,\ghat}{\lambdamin}(\lambdamax(0),\lambdamin(0)) \cdot \partial^+ \lambdamin(0)\\
		&= \pdd{\,\ghat}{\lambdamax}(c,c) \cdot \partial^+ \lambdamax(0) + \pdd{\,\ghat}{\lambdamin}(c,c) \cdot \partial^- \lambdamax(0)\notag
	\intertext{and}
		\partial^- p(0) &= \partial^- \ghat(\lambdamax(t),\lambdamin(t)) \big|_{t=0}\notag\\
		&= \pdd{\,\ghat}{\lambdamax}(\lambdamax(0),\lambdamin(0)) \cdot \partial^- \lambdamax(0) + \pdd{\,\ghat}{\lambdamin}(\lambdamax(0),\lambdamin(0)) \cdot \partial^- \lambdamin(0)\\
		&= \pdd{\,\ghat}{\lambdamax}(c,c) \cdot \partial^- \lambdamax(0) + \pdd{\,\ghat}{\lambdamin}(c,c) \cdot \partial^+ \lambdamax(0)\,.\notag
	\end{align}
	Therefore, the inequality $\partial^+ p(0)\geq\partial^- p(0)$ is equivalent to
	\begin{equation}\label{eq:mainResultProofInequalityWithLambdaMax}
		\pdd{\,\ghat}{\lambdamax}(c,c) \cdot (\partial^+ \lambdamax(0) - \partial^- \lambdamax(0))
		\;\geq\;
		\pdd{\,\ghat}{\lambdamin}(c,c) \cdot (\partial^+ \lambdamax(0) - \partial^- \lambdamax(0))\,.
	\end{equation}
	Since $\partial^+ \lambdamax(0) \geq \partial^- \lambdamax(0)$ according to Lemma \ref{lemma:singularValueDerivativePropertiesPlanar}, inequality \eqref{eq:mainResultProofInequalityWithLambdaMax} can be further simplified to
	\begin{equation}\label{eq:mainResultProofInequalitySimplified}
		\pdd{\,\ghat}{\lambdamax}(c,c) \geq \pdd{\,\ghat}{\lambdamin}(c,c)\,.
	\end{equation}
	Recall from Section \ref{section:BEinequalities} (cf.\ Proposition \ref{proposition:knowlesSternbergOriginal}) that the Baker-Ericksen inequalities,%
	\footnote{%
		Note that, although we provide the explicit computation again due to the change in notation for the planar case, the statement \eqref{eq:mainResultProofInequalitySimplified} also follows directly from \eqref{eq:partialDerivativeOrderingImpliedByBEinequalities} with $n=2$, $\lambda=c$, $\imin[]=1$ and $\imax[]=2$.%
	}
	which in the planar case can be expressed as
	\[
		\lambda_1\cdot\pdd{\,\ghat}{\lambdamax}(\lambda_1,\lambda_2)
		\;\geq\; \lambda_2\cdot\pdd{\,\ghat}{\lambdamin}(\lambda_1,\lambda_2)
		\,,
	\]
	hold for all $\lambda_1,\lambda_2\in\Rp$ with $\lambda_1>\lambda_2$ due to the assumed ellipticity. Since $\ghat\in C^1(\Oset[2])$ by assumption, the continuity of the derivative of $\ghat$ yields
	\[
		c\cdot\pdd{\,\ghat}{\lambdamax}(c,c)
		\;=\; \lim_{\lambda\upto c}\, c\cdot\pdd{\,\ghat}{\lambdamax}(c,\lambda)
		\;\geq\; \lim_{\lambda\upto c}\, \lambda\cdot\pdd{\,\ghat}{\lambdamin}(c,\lambda)
		\;=\; c\cdot\pdd{\,\ghat}{\lambdamin}(c,c)
		\,,
	\]
	which establishes \eqref{eq:mainResultProofInequalitySimplified} and thereby proves the theorem.
\end{proof}

\begin{remark}
\label{remark:strictRankOneConvexityPlanarCase}
	If, under the assumptions of Theorem \ref{theorem:mainResultPlanar}, the function $W$ is also \emph{strictly elliptic} on $\GLp(2)\setminus\CSO(2)$, i.e.\ if strict inequality holds in \eqref{eq:positiveSecondDerivativeRankOneDirection} for each $F\in\CSO(2)$, then it is easy to see that $W$ is \emph{strictly rank-one convex} \cite{agn_neff2016injectivity,agn_schweickert2019nonhomogeneous} on $\GLp(2)$ as well, i.e.\ that strict inequality holds in \eqref{definition:rankOneConvexity}.
\end{remark}

\subsection{The Knowles-Sternberg ellipticity criterion}
\label{section:knowlesSternberg}

One of the most important criteria for ellipticity of planar isotropic functions is the following result by Knowles and Sternberg \cite{knowles1976failure,knowles1978failure}.

\begin{proposition}[{Knowles and Sternberg \cite{knowles1976failure,knowles1978failure}, cf.\ \cite[p.~308]{silhavy1997mechanics}}]
\label{proposition:knowlesSternbergOriginal}
	Let $W\col\GLp(2)\to\R$ be an objective-isotropic function with  $W(F)=g(\lambda_1,\lambda_2)$ for all $F\in\GLp(2)$ with singular values $\lambda_1,\lambda_2$, where \mbox{$g\col\Rp^2\to\R$.}
	If $W$ is two-times continuously differentiable at $F\in\GLp(2)$ with singular values $\lambda_1,\lambda_2$, then $W$ is Legendre-Hadamard elliptic at $F$ if and only if $g$ satisfies the following conditions at $(\lambda_1,\lambda_2)$:
	\begin{alignat*}{2}
		\text{i)}&\qquad g_{11}\;\geq\; 0 \qquad\text{and}\qquad g_{22}\;\geq\;0\,,\\
		\text{ii)}&\qquad \frac{\lambda_1\.g_1-\lambda_2\.g_2}{\lambda_1-\lambda_2}\;\geq\; 0
			&&\qquad\text{if }\; \lambda_1\neq \lambda_2\,, \hspace*{.21\textwidth}\\
		\text{iii)}&\qquad g_{11}-g_{12}+\frac{g_1}{\lambda_1}\;\geq\; 0 \quad\text{ and }\quad g_{22}-g_{12}+\frac{g_2}{\lambda_2}\;\geq\; 0
			&&\qquad\text{if }\; \lambda_1=\lambda_2\,,\\
		\text{iv)}&\qquad \sqrt{g_{11}\,g_{22}}+g_{12}+\frac{g_1-g_2}{\lambda_1-\lambda_2}\;\geq\; 0
			&&\qquad\text{if }\; \lambda_1\neq \lambda_2\,,\\
		\text{v)}&\qquad \sqrt{g_{11}\,g_{22}}-g_{12}+\frac{g_1+g_2}{\lambda_1+\lambda_2}\;\geq\; 0\,,
	\end{alignat*}
	where $g_i=\pdd{g}{\lambda_i}(\lambda_1,\lambda_2)$ and $g_{ij}=\pdd[2]{g}{\lambda_1\,\partial\lambda_2}(\lambda_1,\lambda_2)$.
	\directqed
\end{proposition}
\noindent
Note that condition ii) is identical to the Baker-Ericksen inequalities \eqref{eq:BEinequalities} for the planar case $n=2$.

Many similar criteria, some of which are based on the work of Knowles and Sternberg, can be found throughout the literature, for example in
\cite[p.~308]{silhavy1997mechanics}, \cite[Proposition 7]{dacorogna01}, \cite[Proposition 6.4]{vsilhavy1999isotropic}, \cite[p.~293]{davies1991simple}, \cite[Theorem 2]{de2012note}, \cite[Theorem 4.2]{simpson1983copositive}, \cite{wang1996reformulation}. Note that in each of these cases, the energy $W$ is assumed to be at least two-times differentiable.

Of course, if $W\in C^2(\GLp(2))$, then Proposition \ref{proposition:knowlesSternbergOriginal} can be used as a necessary and sufficient criterion for global rank-one convexity by applying the appropriate inequalities i)--v) to each $(\lambda_1,\lambda_2)\in\Rp^2$. Employing the Knowles-Sternberg criterion (or related criteria) for this purpose is a common practice in a wide variety of applications \cite{agn_neff2015exponentiatedI,agn_martin2018non}.\footnote{%
	While criteria for rank-one convexity under weakened regularity assumptions are available as well (e.g.\ \cite[Proposition 1]{aubert1995necessary}, \cite[Theorem 6.3]{vsilhavy1999isotropic}, \cite[Proposition 5.5]{vsilhavy2002convexity}, \cite[Theorem 3.2]{vsilhavy2003so} or \cite[Proposition 5.16]{Dacorogna08}), they are generally non-local and more difficult to apply.%
}
Using Theorem \ref{theorem:mainResultPlanar}, we can now reduce the regularity requirements imposed by Proposition \ref{proposition:knowlesSternbergOriginal}.

\begin{proposition}
\label{proposition:globalknowlesSternbergImprovement}
	Let $W\col\GLp(2)\to\R$ be an objective-isotropic function with  $W(F)=\ghat(\lambda_1,\lambda_2)$ for all $F\in\GLp(2)$ with singular values $\lambda_1\geq\lambda_2$, where $\ghat\col\Oset[2]\to\R$ is a real-valued function on the set $\Oset[2]=\{(x,y)\in\Rp^2 \setvert x\geq y\}$.
	If $\ghat\in C^1(\Oset[2])\cap C^2(\intOset[2])$ and
	\begin{alignat*}{2}
		\text{i)}&\qquad \ghat_{11}\geq 0 \qquad\text{and}\qquad \ghat_{22}\geq0\,,\qquad\qquad %
		&&\text{ii)}\qquad \frac{\lambda_1\.\ghat_1-\lambda_2\.\ghat_2}{\lambda_1-\lambda_2}\geq 0\,,\\
		\text{iii)}&\qquad \sqrt{\ghat_{11}\,\ghat_{22}}+\ghat_{12}+\frac{\ghat_1-\ghat_2}{\lambda_1-\lambda_2}\geq 0\,,\qquad\qquad
		&&\text{iv)}\qquad \sqrt{\ghat_{11}\,\ghat_{22}}-\ghat_{12}+\frac{\ghat_1+\ghat_2}{\lambda_1+\lambda_2}\geq 0\,,
	\end{alignat*}
	at each $(\lambda_1,\lambda_2)\in\Rp$ with $\lambda_1>\lambda_2$, where $\ghat_i=\pdd{\,\ghat}{\lambda_i}(\lambda_1,\lambda_2)$ and $\ghat_{ij}=\pdd[2]{\ghat}{\lambda_1\,\partial\lambda_2}(\lambda_1,\lambda_2)$, then $W$ is rank-one convex.
\end{proposition}
\begin{proof}
	According to Theorem \ref{theorem:mainResultPlanar}, $W$ is rank-one convex if and only if $W$ is elliptic at each $F\in\GLp(2)$ with singular values $\lambda_1\neq\lambda_2$. We can therefore apply Proposition \ref{proposition:knowlesSternbergOriginal}, omitting condition iii) and assuming $\lambda_1>\lambda_2$ without loss of generality due to the symmetry of the remaining conditions. In that case, $g\equiv\ghat$ in a neighbourhood of $(\lambda_1,\lambda_2)$, thus conditions i)--iv) in Proposition \ref{proposition:globalknowlesSternbergImprovement} are equivalent to conditions i), ii), iv) and v) in Proposition \ref{proposition:knowlesSternbergOriginal}, which are in turn equivalent to the ellipticity of $W$ at $F$.
\end{proof}

\begin{remark}
	If the representation $g$ of $W$ in terms of unordered singular values is twice continuously differentiable, then Proposition \ref{proposition:globalknowlesSternbergImprovement} is applicable with $\ghat=g$ on $\Oset$. Note also that condition iii) in the original Knowles-Sternberg criterion (Proposition \ref{proposition:knowlesSternbergOriginal}) is redundant as far as global rank-one convexity is concerned, as was (for the regular case $g\in C^2(\Rp^n)$) already observed by Dacorogna \cite[Proposition 7]{dacorogna01}. As a pointwise criterion for ellipticity at some $F\in\GLp(2)$, however, it cannot be omitted if $F$ has two identical singular values.
\end{remark}

\section{Rank-one convexity and ellipticity for isotropic functions in arbitrary dimension}
\label{section:generalCase}

In order to extend Theorem \ref{theorem:mainResultPlanar} to the arbitrary-dimensional case, we will first need to generalize some results from the previous section.
First, in contrast to the planar case, matrices with non-simple singular values are not necessarily unique or even isolated along rank-one connection in $\GLpn$ for $n>2$.
For example, consider $F=\id=\diag(1,1,1)$ and $H=\diag(1,0,0)$; then $F\in\GLp(3)$ and $H\in\R^{3\times3}$ with $\rank(H)=1$ and $F+H\in\GLp(3)$, but $F+tH$ has non-simple singular values for each $t\in[0,1]$. The following lemma, however, will prove to be a sufficient replacement.

\begin{lemma}
\label{lemma:isolatedPointsOfIrregularity}
	Let $F\in\GLpn$ and $H\in\Rnn$ such that $\rank(H)=1$ and $F+H\in\GLpn$. Then there exist finitely many $t_1,\dotsc,t_m\in(0,1)$ such that the mapping $t\mapsto \lambdahat(F+tH)$ of $t$ to the vector of ordered singular values $\lambdahat_1(F+tH)\geq\dotsc\geq\lambdahat_n(F+tH)$ is two-times continuously differentiable on $(0,1)\setminus\{t_1,\dotsc,t_m\}$.
\end{lemma}
\begin{proof}
	Again, with $\mu_1,\dotsc,\mu_n$ as in Corollary \ref{corollary:singularValuesAnalyticFunctions}, we observe that for $i,j\in\{1,\dotsc,n\}$, the analytic functions $\mu_i-\mu_j$ are either constant or have isolated roots. We can therefore choose $0=t_0<t_1<\dotsc<t_m<t_{m+1}=1$ such that for each $k\in\{0,\dotsc,m\}$, either $\mu_i\equiv\mu_j$ or $\mu_i(t)\neq\mu_j(t)$ for all $i,j\in\{1,\dotsc,n\}$ and all $t\in(t_k,t_{k+1})$. In particular, for each of these intervals, there exists a (fixed) permutation $\muhat_1,\dotsc,\muhat_n$ of $\mu_1,\dotsc,\mu_n$ such that $\muhat_1(t)\geq\dots\geq\muhat_n(t)$ and thus $\muhat_i(t)=\lambdahat_i(F+tH)$ for all $i\in\{1,\dotsc,n\}$ and all $t\in(t_k,t_{k+1})$. Thus the mapping $t\mapsto \lambdahat(F+tH)$ is analytic and, in particular, twice continuously differentiable on $(0,1)\setminus\{t_1,\dotsc,t_m\}$.
\end{proof}

Our main result will require a generalization of Lemma \ref{lemma:singularValueDerivativePropertiesPlanar} to the arbitrary-dimensional case as well. However, since the multiplicity of a singular value might be larger than two, the property \eqref{eq:planarOneSidedDerivativesEqualityInequality} will need to be replaced by an analogous relation between the one-sided partial derivatives. Furthermore, we need to account for the case where different singular values each occur multiple times. In order to treat these occurrences separately, we introduce the following notation: Let $\lambdahat_1\geq\ldots\geq\lambdahat_n$. Then for each $\lambda\in\{\lambdahat_1,\dotsc,\lambdahat_n\}$, there exists a minimal index $\imin\in\{1,\dotsc,n\}$ and a maximal index $\imax\in\{1,\dotsc,n\}$ such that $\lambdahat_{\imin}=\lambdahat_{\imax}=\lambda$. In particular, $\lambdahat_i=\lambda$ if and only if $i\in \iset\colonequals\{\imin,\dotsc,\imax\}$ due to the ordering of $(\lambdahat_1,\dotsc,\lambdahat_n)$. We can therefore write the set $\{1,\dotsc,n\}$ as the \emph{disjoint} partition
\begin{equation}\label{eq:disjointPartitionIndexSet}
	\{1,\dotsc,n\}
	\;\;= \bigcup_{\lambda\in\{\lambdahat_1,\dotsc,\lambdahat_n\}} \!\!\{\imin,\dotsc,\imax\}
	\;\;= \bigcup_{\lambda\in\{\lambdahat_1,\dotsc,\lambdahat_n\}} \!\!\iset
	\,.
\end{equation}

We briefly remark that using the above notation, the ordering of the partial derivatives expressed by \eqref{eq:partialDerivativeOrderingImpliedByBEinequalities} under the assumption of ellipticity can be stated as follows.

\begin{lemma}
\label{lemma:partialDerivativeOrderingImpliedByBEinequalities}
	Let $W\col\GLpn\to\R$ be elliptic at $F\in\GLpn$, and for $\lambda\in\{\lambdahat_1(F),\dotsc,\lambdahat_n(F)\}$, let $\imin,\imax\in\{1,\dotsc,n\}$ such that $\lambdahat_i(F)=\lambda$ if and only if $i\in\iset=\{\imin,\dotsc,\imax\}$. Then
	\begin{equation*}
		\pdd{\,\ghat}{\lambdahat_i}(\lambdahat_1(F),\dotsc,\lambdahat_n(F))
		\;\geq\; \pdd{\,\ghat}{\lambdahat_j}(\lambdahat_1(F),\dotsc,\lambdahat_n(F))
		\qquad\tforall i,j\in\iset%
		\;\;\;\twith i\leq j
		\,.
		\qedhere\directqed
	\end{equation*}
\end{lemma}

While the following generalization of Lemma \ref{lemma:singularValueDerivativePropertiesPlanar} is more involved than its planar counterpart, it will be sufficient for our purpose.

\begin{lemma}
\label{lemma:singularValueDerivativePropertiesGeneral}
	Let $F,H\in\Rnn$, and let $I\subset\R$ be an interval such that $F+tH\in\GLpn$ for all $t\in I$. Then the mapping
	\[
		\lambdahat\col I\to\R\,,\qquad \lambdahat(t) = (\lambdahat_1(F+tH),\dotsc,\lambdahat_n(F+tH))
	\]
	of $t$ to the ordered singular values $\lambdahat_1(F+tH),\dotsc,\lambdahat_n(F+tH)$ of $F+tH$ is both left-differentiable and right-differentiable on the interior of $I$.
	Moreover, for any $t\in I$ and any $\lambda\in\lambdahat(t)$,%
	\footnote{%
		Here and henceforth, we use $\lambda\in\lambdahat(t)$ as a shorthand notation for $\lambda\in\{\lambdahat_1(t),\dotsc,\lambdahat_n(t)\}$. Note also that $\imin$ and $\imax$ are well defined due to the ordering of $\lambdahat(t)$.%
	}
	let $\imin,\imax\in\N$ such that $\lambdahat_i(t)=\lambda$ if and only if $i\in\iset=\{\imin,\dotsc,\imax\}$. Then
	\begin{equation}\label{eq:singularValueDerivativesSumEquality}
		\sum_{i\in\iset} \partial^- \lambdahat_i(t) = \sum_{i\in\iset} \partial^+ \lambdahat_i(t)
	\end{equation}
	and
	\begin{equation}\label{eq:singularValueDerivativesPartialSumInequality}
		\sum_{i=\imin}^k \partial^- \lambdahat_i(t) \leq \sum_{i=\imin}^k \partial^+ \lambdahat_i(t)
	\end{equation}
	for all $k\in\{\imin,\dotsc,\imax\}$.
\end{lemma}
\begin{proof}
	Once more, with $\mu_1,\dotsc,\mu_n$ denoting the (unordered) singular value functions given by Corollary \ref{corollary:singularValuesAnalyticFunctions}, we note that the roots of the analytic functions $\mu_i-\mu_j$ are isolated unless $\mu_i\equiv\mu_j$ on $I$ for $i,j\in\{1,\dotsc,n\}$. We can therefore assume that for each $t_0\in I$, there exist permutations $(\muhat_1^-\dotsc,\muhat_n^-)$ and $(\muhat_1^+\dotsc,\muhat_n^+)$ of $(\mu_1\dotsc,\mu_n)$ such that
	\[
		\muhat_1^-(t) \geq \ldots \geq \muhat_n^-(t) \quad\text{ for all }\; t_0\geq t\in I_0
		\qquad\text{and}\qquad
		\muhat_1^+(t) \geq \ldots \geq \muhat_n^+(t) \quad\text{ for all }\; t_0\leq t\in I_0
		\,.
	\]
	in a sufficiently small neighbourhood $I_0$ of $t_0$. Then, in particular, $\lambdahat_i(t)=\muhat_i^-(t)$ if $t\leq t_0$ and $\lambdahat_i(t)=\muhat_i^+(t)$ if $t\geq t_0$ for any $i\in\{1,\dotsc,n\}$.
	
	Now, let $\lambda\in\lambdahat(t_0)$. Then $\muhat_i^-(t_0)=\lambdahat(t_0)=\lambda$ if and only if $i\in\iset$ and, similarly, $\muhat_i^-(t_0)=\lambda$ if and only if $i\in\iset$. Therefore, $(\muhat_{\imin}^-,\dotsc,\muhat_{\imax}^-)$ and $(\muhat_{\imin}^-,\dotsc,\muhat_{\imax}^-)$ must be permutations of one another; more specifically, there exists $J\subset\{1,\dotsc,n\}$ such that
	\begin{equation}\label{eq:singularValuePartialDerivativeLemmaFunctionSets}
		\{ \muhat_i^- \setvert i\in\iset \}
		= \{ \muhat_i^+ \setvert i\in\iset \}
		= \{ \mu_i \setvert i\in J \}
		\,.
	\end{equation}

	Now, due to the regularity of the functions $\mu_i$, the left and right one-sided derivatives of each $\lambdahat_i$ are well defined at $t_0$, and \eqref{eq:singularValueDerivativesSumEquality} follows from the equalities
	\begin{align*}
		\sum_{i\in\iset} \partial^- \lambdahat_i(t_0)
		= \sum_{i\in\iset} \partial^- \muhat_i^-(t_0)
		\overset{\eqref{eq:singularValuePartialDerivativeLemmaFunctionSets}}{=} \sum_{\mathclap{i\in J}} \partial^- \mu_i(t_0)
		= \sum_{\mathclap{i\in J}} \partial^+ \mu_i(t_0)
		\overset{\eqref{eq:singularValuePartialDerivativeLemmaFunctionSets}}{=} \sum_{i\in\iset} \partial^+ \muhat_i^+(t_0)
		= \sum_{i\in\iset} \partial^+ \lambdahat_i(t_0)\,.
	\end{align*}
	It remains to show that \eqref{eq:singularValueDerivativesPartialSumInequality} holds for $t\in I$. We observe that for any $k\in\{1,\dotsc,n\}$,
	\[
		\sum_{i=1}^k \partial^- \lambdahat_i(t)
		= \partial_t^- \, \sum_{i=1}^k \lambdahat_i(t)
		= \partial_t^- \, \norm{F+tH}_k
		\qquad\text{and}\qquad
		\sum_{i=1}^k \partial^+ \lambdahat_i(t)
		= \partial_t^+ \, \sum_{i=1}^k \lambdahat_i(t)
		= \partial_t^+ \, \norm{F+tH}_k\,,
	\]
	where $\norm{X}_k=\sum_{i=1}^k \lambdahat_i(X)$ denotes the \emph{Ky Fan $k$-norm} \cite{fan1951maximum} of $X\in\Rnn$. By virtue of being a norm, the mapping $X\mapsto \norm{X}_k$ is convex on $\Rnn$. Therefore, according to Lemma \ref{lemma:convexityOneSidedDerivatives},
	\begin{equation}\label{eq:kyFanInequality}
		\sum_{i=1}^k \partial^- \lambdahat_i(t) = \partial_t^- \, \norm{F+tH}_k \leq \partial_t^+ \, \norm{F+tH}_k = \sum_{i=1}^k \partial^+ \lambdahat_i(t)\,.
	\end{equation}
	Finally, observe that
	\begin{equation}\label{eq:singularValuePartialDerivativeLemmaPartition}
		\{1,\dotsc,\imin-1\}
		= \bigcup_{\substack{\lambdatilde\in\lambdahat(t)\\\lambdatilde<\lambda}} \{\imin[\lambdatilde],\dotsc,\imax[\lambdatilde]\}
		= \bigcup_{\substack{\lambdatilde\in\lambdahat(t)\\\lambdatilde<\lambda}} \iset[\lambdatilde]
	\end{equation}
	is a disjoint partition of $\{1,\dotsc,\imin-1\}$ and thus
	\begin{align*}
		\sum_{i=\imin}^k \partial^- \lambdahat_i(t)
		= \sum_{i=1}^k \partial^- \lambdahat_i(t) - \!\!\sum_{i=1}^{\imin-1} \partial^- \lambdahat_i(t)
		&\overset{\eqref{eq:kyFanInequality}}{\leq} \sum_{i=1}^k \partial^+ \lambdahat_i(t) - \!\!\sum_{i=1}^{\imin-1} \partial^- \lambdahat_i(t)
		\\&\overset{\eqref{eq:singularValuePartialDerivativeLemmaPartition}}{=} \sum_{i=1}^k \partial^+ \lambdahat_i(t) - \sum_{\substack{\lambdatilde\in\lambdahat(t)\\\lambdatilde<\lambda}}\;\; \sum_{i\in\iset[\lambdatilde]} \partial^- \lambdahat_i(t)
		\\&\overset{\eqref{eq:singularValueDerivativesSumEquality}}{=} \sum_{i=1}^k \partial^+ \lambdahat_i(t) - \sum_{\substack{\lambdatilde\in\lambdahat(t)\\\lambdatilde<\lambda}}\;\; \sum_{i\in\iset[\lambdatilde]} \partial^+ \lambdahat_i(t)
		= \sum_{i=\imin}^k \partial^+ \lambdahat_i(t)
	\end{align*}
	for all $k\in\{\imin,\dotsc,\imax\}$, which shows that \eqref{eq:singularValueDerivativesPartialSumInequality} holds for all $t\in I$ as well.
\end{proof}

\noindent
In addition to the above properties of the singular value mapping, we will also require the following elementary lemma.

\begin{lemma}
\label{lemma:productWithOrderedVector}
	For $d\in\N$, let $x,b\in\R^d$ such that $\sum_{i=1}^k x_i \geq 0$ for all $k\in\{1,\dotsc,d\}$ and $b_1\geq\ldots\geq b_d$. If either of the two conditions
	\begin{itemize}
		\item[i)] $b_i\geq0$ for all $i\in\{1,\dotsc,d\}$, %
		\item[ii)] $\displaystyle \sum_{i=1}^d x_i = 0$\,.
	\end{itemize}
	is satisfied, then $\sum_{i=1}^d x_i\.b_i\geq0$.
\end{lemma}
\begin{proof}
	We prove the sufficiency of i) by induction: For $d=1$,
	the desired inequality obviously follows from the assumptions on $x$ and $b$.
	Now, assume that $\sum_{i=1}^d x_i\.b_i\geq0$ for some $d\in\N$ and every ordered vector $b\in\R^d$ with nonnegative entries. Then, after computing
	\[
		\sum_{i=1}^{d+1} x_i\.b_i
		= \sum_{i=1}^{d+1} x_i\.(b_i-b_{d+1}) + \underbrace{b_{d+1}\vphantom{\sum_{i=1}^{d+1}}}_{\geq0}\cdot\underbrace{\sum_{i=1}^{d+1} x_i}_{\geq0}
		\geq \sum_{i=1}^{d+1} x_i\.(b_i-b_{d+1})
		= \sum_{i=1}^{d} x_i\.(b_i-b_{d+1})\,,
	\]
	we only need to apply the induction hypothesis to the vector $(b_1-b_{d+1},\dotsc,b_d-b_{d+1})\in\R^d$ (which is ordered with nonnegative entries).
	
	If ii) holds, then
	\[
		\sum_{i=1}^d x_i\.b_i
		= \sum_{i=1}^{d-1} x_i\.b_i + x_d\.b_d
		= \sum_{i=1}^{d-1} x_i\.b_i - b_d\,\sum_{i=1}^{d-1} x_i
		= \sum_{i=1}^{d-1} x_i\.(b_i-b_d)
	\]
	for any $d\geq2$, thus we can apply case i) to the vector $(b_1-b_d,\dotsc,b_{d-1}-b_d)$, which is again ordered and has only nonnegative entries.
\end{proof}

We are now ready to state and prove our main result.

\begin{theorem}
\label{theorem:mainResultGeneral}
	Let $W\col\GLpn\to\R$ be an objective and isotropic function with
	\[
		W(F) = \ghat(\lambdahat_1(F),\dotsc,\lambdahat_n(F))
	\]
	for a real-valued function $\ghat\col\Oset\to\R$ on the set $\Oset=\{(x_1,\dotsc,x_n)\in\Rp^n\setvert x_1\geq\ldots\geq x_n\}$, where $\lambdahat_1(F)\geq\ldots\geq\lambdahat_n(F)$ are the ordered singular values of $F$. If $\ghat\in C^2(\Oset)$ and $W$ is Legendre-Hadamard elliptic at each $F\in\GLpn$ with simple singular values $\lambdahat_1(F)>\cdots>\lambdahat_n(F)$, then $W$ is rank-one convex on $\GLpn$.
\end{theorem}
\begin{proof}
	Again, in order to show that the mapping $t\mapsto W(F+tH)$ is convex on $[0,1]$ for any $F\in\GLpn$ and any $H\in\Rnn$ with $\rank(H)=1$ and $F+H\in\GLpn$ under the assumptions of the theorem, let $\lambdahat(t)=(\lambdahat_1(F+tH),\dotsc,\lambdahat_n(F+tH))$ and
	\[
		p\col[0,1]\to\R\,,\qquad p(t) = W(F+tH) = \ghat(\lambdahat(t))\,.
	\]
	Then due to Lemma \ref{lemma:isolatedPointsOfIrregularity}, the function $p$ is twice differentiable on $(0,1)\setminus\{t_1,\dotsc,t_m\}$ for finitely many $t_1,\dotsc,t_N\in(0,1)$, and due to the ellipticity of $W$ on the set of matrices with simple singular values and the assumed $C^2$-regularity of $\ghat$ on $\Oset$ (up to the boundary), it is easy to verify that	$p''(t)\geq0$ for all $(0,1)\setminus\{t_1,\dotsc,t_m\}$.
	
	According to Lemma \ref{lemma:convexityOneSidedDerivatives}, it remains to show that $\partial^- p(t) \leq \partial^+ p(t)$ for all $t\in(0,1)$.
	First, we find
	\[
		\partial^- p(t) = \sum_{i=1}^n \pdd{\,\ghat}{\lambdahat_i}(\lambdahat(t)) \cdot \partial^- \lambdahat_i(t)
		\qquad\text{and}\qquad
		\partial^+ p(t) = \sum_{i=1}^n \pdd{\,\ghat}{\lambdahat_i}(\lambdahat(t)) \cdot \partial^+ \lambdahat_i(t)\,.
	\]
	Now, for $\lambda\in\lambdahat(t)$, choose $\imin,\imax\in\{1,\dotsc,n\}$ as in Lemma \ref{lemma:singularValueDerivativePropertiesGeneral} and let $\iset=\{\imin,\dotsc,\imax\}$.
	Using the disjoint partition of $\{1,\dotsc,n\}$ given in \eqref{eq:disjointPartitionIndexSet}, we find
	\[
		\partial^- p(t) \;=\; \!\!\sum_{\lambda\in\lambdahat(t)}\; \sum_{i\in\iset\vphantom{\lambdahat}}\; \pdd{\,\ghat}{\lambdahat_i}(\lambdahat(t)) \cdot \partial^- \lambdahat_i(t)
		\quad\;\;\text{and}\quad\;\;
		\partial^+ p(t) \;=\; \!\!\sum_{\lambda\in\lambdahat(t)}\; \sum_{i\in\iset\vphantom{\lambdahat}}\; \pdd{\,\ghat}{\lambdahat_i}(\lambdahat(t)) \cdot \partial^+ \lambdahat_i(t)\,.
	\]
	According to Lemma \ref{lemma:singularValueDerivativePropertiesGeneral},
	\[
		\sum_{i\in\iset} \partial^+ \lambdahat_i(t) - \partial^- \lambdahat_i(t) = 0
		\qquad\text{and}\qquad
		\sum_{i=\imin}^k \partial^+ \lambdahat_i(t) - \partial^- \lambdahat_i(t)\geq0
	\]
	for all $t\in(0,1)$, $\lambda\in\lambdahat(t)$ and $k\in\{\imin,\dotsc,\imax\}$,
	and since
	\[
		\pdd{\,\ghat}{\lambdahat_i}(\lambdahat(t))
		\;\geq\; \pdd{\,\ghat}{\lambdahat_j}(\lambdahat(t))
		\qquad\tforall i,j\in \iset=\{\imin,\dotsc,\imax\}
		\quad\twith i\leq j
		\,,
	\]
	according to Lemma \ref{lemma:partialDerivativeOrderingImpliedByBEinequalities}, the vector $\Bigl(\pdd{\,\ghat}{\lambdahat_{\imin}}(\lambdahat(t)),\dotsc,\pdd{\,\ghat}{\lambdahat_{\imax}}(\lambdahat(t))\Bigr)$ is ordered for all $t\in(0,1)$ and each $\lambda\in\lambdahat(t)$ as well. Thus applying Lemma \ref{lemma:productWithOrderedVector} ii) to
	\[
		x_i=\partial^+ \lambdahat_{\imin+i-1}(t) - \partial^- \lambdahat_{\imin+i-1}(t)
		\,,\qquad
		b_i=\pdd{\,\ghat}{\lambdahat_{\imin+i-1}}(\lambdahat(t))
		\qquad\text{and}\qquad
		d=\imax+1-\imin
	\]
	yields
	\[
		\sum_{i\in\iset} \pdd{\,\ghat}{\lambdahat_i}(\lambdahat(t)) \cdot (\partial^+ \lambdahat_i(t) - \partial^- \lambdahat_i(t))
		= \sum_{i=\imin}^{\imax} \pdd{\,\ghat}{\lambdahat_i}(\lambdahat(t)) \cdot (\partial^+ \lambdahat_i(t) - \partial^- \lambdahat_i(t))
		= \sum_{i=1}^m x_i\.b_i
		\geq 0
	\]
	and thus
	\[
		\partial^+ p(t) - \partial^- p(t)
		\;= \sum_{\lambda\in\lambdahat(t)}\; \sum_{i\in\iset}\; \pdd{\,\ghat}{\lambdahat_i}(\lambdahat(t)) \cdot (\partial^+ \lambdahat_i(t) - \partial^- \lambdahat_i(t))
		\;\geq\; 0\,,
	\]
	which concludes the proof.
\end{proof}

\begin{remark}
	Although the regularity assumption $\ghat\in C^2(\Oset)$ is slightly more restrictive than the requirements of Theorem \ref{theorem:mainResultPlanar}, it is still considerably easier to satisfy than $C^2$-regularity on $\GLpn$, as both Examples  \ref{example:operatorNormEnergy} and \ref{example:conformallyInvariantEnergies} demonstrate.
\end{remark}
\subsubsection*{Acknowledgements}
The work of I.D.\ Ghiba has been supported by a grant of the Romanian Ministry of Research and Innovation, CNCS--UEFISCDI, project number PN-III-P1-1.1-TE-2019-0397, within PNCDI III.

\section{References}
\footnotesize
\printbibliography[heading=none]

\end{document}